\documentclass[12pt]{article}
 \usepackage{amsthm,amsmath,amssymb}
 \usepackage{color}
 
   \usepackage[all]{xy}
  \usepackage{graphicx}

\theoremstyle{plain}
\newtheorem{Thm}{Theorem}

\newtheorem{Def}[Thm]{Definition}

\newtheorem{remark}[Thm]{Remark}

\newtheorem{theorem}[equation]{Theorem}

\newcommand{\twotwo}{\mathbf{2+2}}
\newcommand{\threeone}{\mathbf{3+1}}
\newcommand{\fourone}{\mathbf{4+1}}
\newcommand{\none}{\mathbf{n+1}}

\date{August 30, 2017 }

\title{A Simple Proof  Characterizing Interval Orders with Interval Lengths between 1 and $k$}

\author{Simona Boyadzhiyska\thanks{This work was supported by a Jerome A. Schiff Fellowship at Wellesley College. }\\
\small Berlin Mathematical School \\
 \small Freie Universit\"{a}t Berlin \\
\small Berlin, Germany \\
\small\tt \  s.boyadzhiyska@fu-berlin.de
\and
Garth Isaak\\
\small Department of Mathematics\\
\small Lehigh University\\
\small Bethlehem, PA  18015\\
\small\tt  gi02@lehigh.edu 
\and
Ann  N. Trenk\thanks{
This work was supported by a grant from the Simons Foundation (\#426725, Ann Trenk). } \\
\small Department of Mathematics\\
\small Wellesley College\\
\small Wellesley, MA 02481\\
\small\tt atrenk@wellesley.edu
}

\begin{document}

\maketitle

\begin{abstract} A poset $P= (X, \prec)$ has an interval representation if each $x \in X$ can be assigned a real interval $I_x$ so that $x \prec y$ in $P$ if and only if $I_x$ lies completely to the left of $I_y$.  Such orders are called \emph{interval orders}.  Fishburn  \cite{Fi83,Fi85}
proved that for any positive integer $k$, an interval order has a representation in which all interval lengths are between $1$ and $k$ if and only if the order does  not contain $\mathbf{(k+2)+1}$  as an induced poset.  In this paper, we give a simple proof of this result using a digraph model.

\end{abstract}

\bibliographystyle{plain} 

\section{Introduction}

\subsection{Posets and Interval Orders}
A poset $P$ consists of a set $X$ of \emph{points} and a relation $\prec$ that is irreflexive and transitive, and therefore antisymmetric.  It is sometimes convenient to write $y \succ x$ instead of $x \prec y$.    If $x \prec y$ or $y \prec x$, we say that $x$ and $y$ are \emph{comparable}, and otherwise we say they are \emph{incomparable}, and denote the incomparability by $x \parallel y$.    An \emph{interval representation} of a poset  $P=(X,\prec)$ is an assignment of a closed real interval $I_v$ to each $v\in X$ so that 
$x \prec y$ if and only if   $I_x$   is completely to the left of $I_y$.     A poset with such a representation is called an \emph{interval order}.  It is well-known that  the classes studied in this paper are the same if open intervals are used instead of closed intervals, e.g., see Lemma 1.5 in \cite{GoTr04}.

The poset $\twotwo$ shown in  Figure~\ref{chains-fig}  consists of four elements $\{a,b,x,y\}$ and the only comparabilities are $a \prec x$ and $b \prec y$. The following elegant theorem characterizing interval orders   was anticipated by Wiener in 1914 (see \cite{FiMo92}) and shown by Fishburn \cite{Fi70}:
Poset $P$ is an interval order if and only if it contains no induced $\twotwo$.
    Posets  that have an interval representation in which all intervals are the same length are known as  \emph{unit interval orders} or \emph{semiorders}.    Scott and Suppes \cite{ScSu58} characterize unit interval orders as those posets with no induced $\twotwo$ and no induced $\threeone$.  Figure~\ref{chains-fig} shows the posets $\twotwo$, \  $\threeone$, and $\fourone$.  More generally, the poset $\none$ consists of a chain of $n$ distinct elements $a_1 \prec a_2 \prec \cdots \prec a_n$ and an additional element that is incomparable to each $a_i$.

\begin{figure}
\begin{center}
 \begin{picture}(300,65)(0,15)
\thicklines

\put(20,40){\circle*{5}}
\put(20,70){\circle*{5}}

\put(50,40){\circle*{5}}
\put(50,70){\circle*{5}} 
\put(20,40){\line(0,1){30}}
\put(50,40){\line(0,1){30}}

\put(7,38){$a$}
\put(7,68){$x$}

\put(55,38){$b$}
\put(55,68){$y$}

\put(22,10){$\twotwo$}

\put(105,10){$\threeone$}

\put(196,10){$\fourone$}

\put(120,35){\circle*{5}}
\put(120,55){\circle*{5}}
\put(120,75){\circle*{5}}
\put(140,55){\circle*{5}}
\put(120,35){\line(0,1){40}}

\put(107,33){$a$}
\put(107,53){$b$}
\put(107,73){$c$}
\put(145,53){$x$}
  
  \put(210,30){\circle*{5}}
\put(210,50){\circle*{5}}
\put(210,70){\circle*{5}}
\put(210,90){\circle*{5}}
\put(230,60){\circle*{5}}
\put(210,30){\line(0,1){60}}

\put(197,28){$a$}
\put(197,48){$b$}
\put(197,68){$c$}
\put(197,88){$d$}
\put(235,58){$x$}

  \end{picture}
  \end{center}

\caption{The posets $\twotwo$, \  $\threeone$, and $\fourone$.}

\label{chains-fig}
 \end{figure}

In this paper, we consider an intermediate class between the extremes of   interval orders (no restrictions on interval lengths) and unit interval orders (all intervals the same length).    In particular, we allow interval lengths to range from 1 to $k$, where $k$ is a positive integer.  Fishburn \cite{Fi83} characterizes this class as those posets with no induced $\twotwo$ and no induced $\mathbf{(k+2)+1}$, generalizing the result of Scott and Suppes.   In fact, Fishburn characterizes those posets that have an  interval representation by intervals whose lengths are between $m$ and $n$ for any relatively prime integers $m,n$ in terms of  what he calls \emph{picycles}. The proof is technical, and it does not immediately yield a forbidden poset characterization in the general case.
We use a digraph model from Isaak~\cite{Is09} to give a shorter and more accessible proof in the case $m=1, n=k$. Our digraph model and the equivalence of statements (1) and   (3) in Theorem~\ref{lengths1tok} can easily be extended to general $m,n$.  It is also natural to consider allowing the interval lengths to vary between 1 and any real value. Fishburn and Graham \cite{FiGr85} study the classes $C(\alpha)$ of interval graphs that have  a representation  by intervals with lengths between $1$ and $\alpha$ for any real $\alpha\geq 1$, showing that the points where $C(\alpha)$ expands are the rational values of $\alpha$.  
The problem of characterizing posets that have an interval representation in which the possible interval lengths come from a discrete set (rather than from an interval)  is more challenging, and we consider two variants of this question in \cite{BoIsTr17}. 

\subsection{Digraphs and Potentials}

A {\em directed graph}, or {\em digraph}, is a pair $G=(V,E)$, where $V$ is a finite set of {\em vertices}, and $E$ is a set of ordered pairs $(x,y)$ with $x,y\in V$, called {\em arcs}.   A \emph{weighted digraph} is a digraph in which each arc $(x,y)$ is assigned a real number weight $w_{xy}$.  We sometimes denote the arc $(x,y)$ by $x\rightarrow y$, and in a weighted digraph, by $x \xrightarrow{w_{xy}}y$.  A \emph{potential function}  $p:V\rightarrow \mathbb{R}$, defined on the vertices of a weighted digraph, is a function satisfying $p(y) - p(x) \leq w_{xy}$ for each arc  $(x,y)$.  Theorem~\ref{potential-no-neg} is a well-known result that   specifies precisely which digraphs have potential functions.

   A \emph{cycle} in digraph $G$  is a subgraph  with vertex set  $\{ x_1, x_2, x_3,  \dots, x_t \}$ and arc set  $\{ (x_{i}, x_{i+1}): 1 \le i \le t-1\} \cup \{(x_t,x_1)\}$.
In a weighted digraph, the \emph{weight} of  cycle $C$, denoted by $wgt(C)$,  is the sum of the weights of the arcs
 of $C$.    A cycle with negative weight is called a \emph{negative cycle}.  The following theorem is well-known, see Chapter 8 of \cite{Sc03} for example, and we  provide a proof in \cite{BoIsTr17}.

\begin{theorem}
A weighted digraph has a potential function if and only if it contains no negative cycle.
\label{potential-no-neg}
\end{theorem}
 \section{Orders with a $[1,k]$-Interval Representation}

We say that poset $P$ has an $[a,b]$-interval representation if it has a representation by intervals whose lengths are between $a$ and $b$ (inclusive).  When $a=b>0$, the  posets with such a representation are the unit  interval orders.      Because representations can be scaled, for any $b>0$, all interval orders have a $[0,b]$-interval representation.  
 This motivates us to consider the lower bound $a=1$, and in particular, posets that have a $[1,k]$-interval representation where  $k$ is a positive integer.    Fishburn characterized this class in \cite{Fi83} by showing the equivalence of (1) and (2) in Theorem~\ref{lengths1tok};  however, the proof is quite technical.  
Using the framework in \cite{Is09}, we construct a weighted digraph $G_{P,k}$ associated with poset $P$ and show that $P$ has a $[1,k]$-interval representation      if and only if  $G_{P,k}$ has no negative cycle.   This allows for a more  accessible proof of Theorem~\ref{lengths1tok}.   We choose the value of $\epsilon$  appearing as a weight in $G_{P,k}$  so that $0 < \epsilon < \frac{1}{2|X|}$.

\begin{Def} {\rm
Let $P=(X,\prec)$ be a partial order. Define $G_{P,k}$ to be the weighted digraph with vertices $\{\ell_x,r_x\}_{x\in X}$ and the following arcs:
\begin{itemize}
\item $(\ell_y, r_x)$ with weight $-\epsilon$ for all $x,y\in X$ with $x\prec y$,
\item $(r_x,\ell_y)$ with weight $0$ for all $x,y\in X$ with $x||y$,
\item $(r_x,\ell_x)$ with weight $-1$ for all $x\in X$,
\item $(\ell_x,r_x)$ with weight $k$ for all $x\in X$.
\end{itemize}
}
\label{GPk-def}
\end{Def}

It is helpful to think of the arcs of $G_{P,k}$ as  coming in two categories:  $\ell \to r$ and $r \to \ell$.  We list the arcs by category for easy reference.

\begin{center}
\begin{tabular}{|c|c|c|c|}
\hline
Type & Arc & Weight & $x,y$ Relation \\ 
\hline
$\ell \to r$ & $(\ell_y,r_x)$  & $-\epsilon $ & $y \succ x$ \\
\hline
 & $(\ell_x,r_x)$  & $k $ &   \\
\hline
$r \to \ell$ & $(r_x,\ell_y)$  & $0 $ & $x \parallel y$ \\
\hline
 & $(r_x,\ell_x)$  & $-1 $ &   \\
\hline
\end{tabular}

\end{center}
 Any negative   cycle in $G_{P,k}$ with a minimum number of arcs will have at most $2|X|$ arcs since $G_{P,k}$ has $2|X|$ vertices.  Since $\epsilon$ satisfies $0 < \epsilon < \frac{1}{2|X|}$, the
  arcs of weight $- \epsilon$ will have combined weight $w$, where $-1 < w \le 0$.  We record a consequence of this observation in the following remark.

\begin{remark} {\rm
If $C$ is a negative weight cycle in $G_{P,k}$ containing the minimum number of arcs, then $C$ contains at least $k$ arcs of weight $-1$ for every arc of weight $k$.  }
\label{remark-k}
\end{remark}

 \begin{theorem} 
Let $P=(X,\prec)$ be a partial order and let $k\in \mathbb{Z}_{\geq 1}$. The following are equivalent:

\begin{enumerate}
    \item $P$ has a $[1,k]$-interval representation.
    \item $P$ contains no induced $\mathbf{2+2}$ or $\mathbf{(k+2)+1}$.
    \item The weighted digraph $G_{P,k}$ contains no negative cycle.
\end{enumerate}
\label{lengths1tok}
\end{theorem}

\begin{proof}

\noindent $(1) \Rightarrow (3)$   Suppose that $P$ has an interval representation \mbox{$\mathcal{I} = \{I_x\}_{x\in X}$}, where $I_x = [L(x), R(x)]$,  and for each $x \in X$ we have $1 \le |I_x| \le k$.    Choose
 $\epsilon = \min\{\frac{1}{2|X| + 1},\delta \}$,  where  $\delta$ is   the smallest distance between unequal endpoints in the representation $\mathcal{I}$.  By the definition of an interval representation and the conditions on the interval lengths, we have

\begin{enumerate}
\item $R(x) - L(y) \leq -\epsilon$ for all $x,y\in X$ with $x\prec y$,
\item $L(y) - R(x) \leq 0$ for all $x,y\in X$ with $x||y$,
\item $L(x) - R(x)\leq -1$ for all $x\in X$,
\item $R(x) - L(x) \leq k$ for all $x\in X$.
\end{enumerate}

Now define  the function $p$ on the vertex set of $G_{P,k}$ as follows.  For each $x \in X$ let  
 $p(r_x) = R(x)$ and $p(\ell_x) = L(x)$.  So  $p$ satisfies
\begin{enumerate}
\item[(a)] $p(r_x) - p(l_y) \leq -\epsilon$ for all $x,y\in X$ with $x\prec y$,
\item[(b)] $p(l_y) - p(r_x) \leq 0$ for all $x,y\in X$ with $x||y$,
\item[(c)] $p(l_x) - p(r_x)\leq -1$ for all $x\in X$,
\item[(d)] $p(r_x) - p(l_x) \leq k$ for all $x\in X$.
\end{enumerate}
 
 Thus, for all $(u,v) \in E(G_{P,k})$, we have $p(v) - p(u)\leq w_{uv}$.  
Hence $p$ is a potential function on $G_{P,k}$ and by 
Theorem~\ref{potential-no-neg}, $G_{P,k}$ has no negative cycle.

\medskip

\noindent$(3) \Rightarrow (1)$ 
Given $G_{P,k}$ has no negative cycle, by Theorem~\ref{potential-no-neg}, there exists a potential function $p$ on $G_{P,k}$, and by definition, $p$ satisfies (a), (b), (c), (d).  For each $x \in X$, let $L(x) = p(\ell_x)$ and $R(x) = p(r_x)$.  By (c) we know $L(x) + 1 \le R(x)$, so $I_x = [L(x),R(x)]$ is indeed an interval with $|I_x| \ge 1$.  By   (d), the length of interval $I_x$ satisfies $  |I_x| \le k$, and by (a) and (b), $x \prec y$ in $P$ if and only if $R(x) < L(y)$.  Thus the set of intervals $\{I_x\}_{x \in X}$ forms a representation of $P$ in which each interval has length between 1 and $k$.

\medskip
\noindent$(3) \Rightarrow (2)$ 
If $P$ contains an induced $\mathbf{2+2}$, denoted by  $(x \succ a)||(y \succ b)$, then $\ell_x \xrightarrow{-\epsilon}r_a\xrightarrow{0}\ell_y\xrightarrow{-\epsilon}r_b\xrightarrow{0}\ell_x$ is a  cycle in $G_{P,k}$ with weight $-2\epsilon$. Similarly, if $P$ contains an induced $\mathbf{(k+2)+1}$, denoted by  $x \parallel  (a_{k+2}  \succ a_{k+1 }   \succ \cdots \succ a_2 \succ a_1) $, then $G_{P,k}$ contains the cycle

\noindent
\scalebox{0.9}{\parbox{\linewidth}{ 
 $$r_x\xrightarrow{0}\ell_{a_{k+2}}\xrightarrow{-\epsilon}r_{a_{k+1}}\xrightarrow{-1}\ell_{a_{k+1}}\xrightarrow{-\epsilon}r_{a_k} \xrightarrow{-1}\ell_{a_k}  \xrightarrow{-\epsilon}
  \cdots    \xrightarrow{-\epsilon} r_{a_2}\xrightarrow{-1}\ell_{a_2}\xrightarrow{-\epsilon}r_{a_1}\xrightarrow{0}l_x\xrightarrow{k}r_x,$$
}}

\noindent whose weight is $(-1)k + k + (-\epsilon)(k+1) < 0.$  In either case, we obtain a negative cycle in $P$, a contradiction.

\medskip
\noindent $(2) \Rightarrow (3)$ 
Now assume $P$ contains no induced $\mathbf{2+2}$ or $\mathbf{(k+2)+1}$.
For a contradiction, assume that $G_{P,k}$ contains a negative cycle, and 
let $C$ be a negative   cycle  in $G_{P,k}$ containing the minimum number of arcs.  By definition of $G_{P,k}$, the arcs in $C$ must alternate between arcs of   type $\ell \rightarrow r$ and arcs of type $r \rightarrow \ell$, thus $C$ has the form  \mbox{$\ell_{x_1}\rightarrow r_{x_2} \rightarrow \ell_{x_3} \rightarrow \dots \rightarrow r_{x_n} \rightarrow \ell_{x_1}$} for some $x_1,x_2,\dots,x_n \in X$, not necessarily distinct.  Since no cycle in $G_{P,k}$ contains exactly two arcs, we know $n \ge 4$.  Furthermore, since vertices of a cycle are distinct, we know that $x_i \neq x_{i+2}$ for $1 \le i \le n$, where the indices are taken modulo $n$.

Next we show $wgt(C) \le -2\epsilon$.   Since $x_i \neq x_{i+2}$ for $1 \le i \le n$ (indices taken modulo $n$), the arcs  of $C$ immediately before and after a weight $k$ arc must have weight 0.  If $C$ has at most one arc of weight $-\epsilon$ , then the remaining $\ell \to r$ arcs have weight $k$, resulting in a positive weight for $C$, a contradiction.   Thus $C$ contains at least 
two arcs  of weight $-\epsilon$, and   Remark~\ref{remark-k} implies that  $wgt(C) \le -2\epsilon$.

We next claim that $C$ does not contain a segment of three consecutive arcs of weights $-\epsilon, 0, -\epsilon$.  For a contradiction, suppose $C$ contains the segment 
\mbox{$S_1:  \ell_{a}\xrightarrow{-\epsilon} r_{b} \xrightarrow{0}\ell_{c}  \xrightarrow{-\epsilon}r_{d}$}.  Then by the definition of $G_{P,k}$, we have $a \succ b$,  \ $b \parallel c$, \ and $c \succ d$.  If $d \succ a$, we get $c \succ d \succ a \succ b$, contradicting $ b \parallel c$.  If $a \parallel d$, then the elements $a,b,c,d$ induce in $P$ the poset 
$\twotwo$, a contradiction.  Otherwise,  $a \succ d$ and we can replace the segment  $S_1$ by 
\mbox{$\ell_{a}\xrightarrow{-\epsilon} r_{d}$} to yield a shorter cycle $C'$ with $wgt(C') = wgt(C) + \epsilon \le -2\epsilon + \epsilon = -\epsilon < 0$.  This contradicts the minimality of $C$.  

We now consider two cases depending on whether or not $C$ contains an arc of weight $k$.

\noindent
{\bf Case 1:  $C$ has no arc of weight $k$.}
In this case, $C$ alternates between arcs with weight $-\epsilon$ and arcs with weight in the set $\{0,-1\}$.    Since $C$ has at least four arcs and no segment of the form $(-\epsilon, 0, -\epsilon)$, there must be an arc of weight $-1$.  Without loss of generality, choose a starting point for $C$ so that it begins with the segment \mbox{$S_2: \ell_{x_1}\xrightarrow{-\epsilon}  r_{x_2} \xrightarrow{-1} \ell_{x_3} \xrightarrow{-\epsilon}  r_{x_4}.$}   By the definition of $G_{P,k}$ we have $x_1 \succ x_2 = x_3 \succ x_4$, so $x_1 \succ x_4$.  Replace  segment $S_2$ by \mbox{$\ell_{x_1}\xrightarrow{-\epsilon}   r_{x_4}$}  to obtain a  cycle $C'$ whose weight is also negative since it contains no arcs of weight $k$.  Since $C'$ has fewer arcs than $C$, this contradicts the minimality of $C$.

\noindent
{\bf Case 2:  $C$ contains an arc of weight $k$.}
By Remark~\ref{remark-k}, there is a segment of $C$ that starts with an arc of weight $k$  and has at least $k$ arcs of weight $-1$ before the next arc of weight $k$.  Without loss of generality, we can choose the starting point of $C$ so that it begins with the segment
\mbox{$\ell_{x_1}\xrightarrow{k}  r_{x_2} \xrightarrow{} \ell_{x_3} \xrightarrow{-\epsilon}  r_{x_4} \xrightarrow{}  \cdots \xrightarrow{-\epsilon}  r_{x_{2k}} \xrightarrow{}  \ell_{x_{2k+1}} $}.
If the arc $(r_{x_2},\ell_{x_3})$ has weight $-1$, then $x_1 = x_2 = x_3$,  a contradiction since $x_1 \neq x_3$.     Thus, the arc $(r_{x_2},\ell_{x_3})$  has weight 0 and $C$ begins with the segment \mbox{$\ell_{x_1}\xrightarrow{k}  r_{x_2} \xrightarrow{0} \ell_{x_3} \xrightarrow{-\epsilon}  r_{x_4} .$}  
 
 If any of the next $k$ arcs of the type $r \rightarrow \ell$ on $C$ had weight 0, then $C$ would contain a segment of the form $(-\epsilon, 0, -\epsilon)$, contradicting our earlier claim. 
  Thus each of these arcs has weight $-1$ and $C$ starts with the following segment:
 \mbox{$ \ell_{x_1}\xrightarrow{k}  r_{x_2} \xrightarrow{0} \ell_{x_3} \xrightarrow{-\epsilon}  r_{x_4} \xrightarrow{-1} \ell_{x_5}  \xrightarrow{-\epsilon}  r_{x_6} \xrightarrow{-1} \cdots  \xrightarrow{-\epsilon}  r_{x_{2k+2}} \xrightarrow{-1} \ell_{x_{2k+3}}.$}  
 
 \smallskip
 By the definition of $G_{P,k}$, we have the following relations in $P$:
 
 \noindent
 $x_1 = x_2 \parallel x_3 \succ x_4 = x_5 \succ x_6 = x_7 \succ  \cdots  = x_{2k+1} \succ x_{2k+2} = x_{2k+3}$.

     If $x_1 = x_{2k+3}$,  then by transitivity, $x_1 \prec x_3$, contradicting the relation $x_1 = x_2 \parallel x_3$.    Thus $C$ contains at least two more arcs  $(\ell_{x_{2k+3}} , r_{x_{2k+4}})$ and $(r_{x_{2k+4}}, \ell_{x_{2k+5}})$.  If arc $(\ell_{x_{2k+3}} , r_{x_{2k+4}})$ had weight $k$, then $x_{2k+2} = x_{2k+3} =x_{2k+4}$,   a contradiction since $x_{2k+2} \neq x_{2k+4}$.     Thus arc $(\ell_{x_{2k+3}} , r_{x_{2k+4}})$ has weight $- \epsilon$,  and $x_{2k+3} \succ x_{2k+4}$ in $P$, and $C$ starts with the following segment: 
    
    \noindent 
    \scalebox{0.95}{\parbox{\linewidth}{ 
     $$S:  \ell_{x_1}\xrightarrow{k}  r_{x_2} \xrightarrow{0} \ell_{x_3} \xrightarrow{-\epsilon}  r_{x_4} \xrightarrow{-1} \ell_{x_5}  \xrightarrow{-\epsilon}  r_{x_6} \xrightarrow{-1} \cdots  \xrightarrow{-\epsilon}  r_{x_{2k+2}} \xrightarrow{-1} \ell_{x_{2k+3}} \xrightarrow{-\epsilon}  r_{x_{2k+4}} .$$ }  }
     
    Finally, we consider the relation between $x_1$ and $x_{2k+4}$ in $P$.   If  $x_1 \prec x_{2k+4}$, then by transitivity, $x_1 \prec x_3$, a contradiction.    If $x_1 \succ x_{2k+4}$, we can replace segment $S$ by $\ell_{x_1}\xrightarrow{-\epsilon}  r_{x_{2k+4}} $ to obtain a shorter cycle $C'$ in $G_{P,k}$.  As noted earlier, the combined weight of the arcs of $C$ that have  weight $- \epsilon$   is strictly greater than $-1$, so $C'$ also has negative weight, contradicting the minimality of $C$.   Hence $x_1 \parallel x_{2k+4}$ and 
 the  $k+3$ elements in the set  $\{x_1, x_3, x_5, \ldots, x_{2k+3}, x_{2k+4}\}$ induce a  $\mathbf{(k+2) + 1}$ in $P$, a contradiction.
\end{proof}

We end by  describing an algorithm that   constructs a $[1,k]$-interval representation of a poset $P$ if one exists   and otherwise produces a forbidden poset, either $\twotwo$ or $\mathbf{(k+2) + 1}$.    Use a standard shortest-paths algorithm such as the Bellman-Ford or the matrix multiplication method on $G_{P,k}$ to compute the weight of a minimum-weight path between each pair of vertices or detect a negative cycle. 
If there is a negative cycle, these algorithms detect one with a minimum number of arcs.  If such a negative cycle exists in $G_{P,k}$, then as in the proof of $(2) \Rightarrow (3)$ of Theorem~\ref{lengths1tok}, either the cycle contains the segment $-\epsilon, 0, -\epsilon$, and a $\twotwo $ is detected in $P$, or else as in Case 2 of that proof, a $\mathbf{(k+2) + 1}$ is detected in $P$.
If there is no negative cycle, 
Theorem~\ref{potential-no-neg} ensures that a potential function  $p$ exists for $G_{P,k}$.  Indeed, setting $p(v)$ to be the minimum weight of a walk ending at $v$ produces a potential function.  As we showed in the proof of $(3) \Rightarrow (1)$, the intervals 
$[p(\ell_x),p(r_x)]$ provide   a $[1,k]$-interval representation of $P$. 
Thus there is a polynomial-time certifying algorithm.

\end{document}